\documentclass[11pt]{article}
\usepackage{amsmath, amsthm, amssymb}
\usepackage{fullpage}
\usepackage{hyperref}

\title{On the Modern Structure of the Gauss-Landau Theorem}
\author{Manuel M., Aguilera \\
\small Mathematical Sciences Department \\
\small University of Puerto Rico, Mayaguez Campus \\
\small \texttt{alex.martinez13@upr.edu}}
\date{}

\theoremstyle{plain}
\newtheorem{theorem}{Theorem}
\newtheorem*{definition}{Definition}
\newtheorem{proposition}[theorem]{Proposition}

\newtheorem*{problem}{Problem}

\begin{document}

\maketitle

\begin{abstract}
We formalize the Gauss-Landau theorem, providing a unified prime factorization approach to computing the GCD and LCM of finite nonzero integer sets. Although commonly used as a heuristic or technique in elementary number theory education, these theorems have not been explicitly formalized or named in the literature. This formalization aims to enhance understanding and facilitate adoption in mathematical instruction and research.
\end{abstract}
\section{Introduction}
The greatest common divisor (GCD) and least common multiple (LCM) are fundamental concepts in elementary number theory with broad applications. Traditional treatments focus on Euclid's algorithm and direct divisibility properties. However, a factorization-based perspective — here named the \emph{Gauss-Landau Theorem} — provides a more structured and powerful approach. This approach uses the unique prime factorization of integers to express the GCD and LCM through the minimum and maximum exponents of the involved primes, respectively. While widely employed as a technique in textbooks and problem-solving, it remains under-recognized as a formal theorem with pedagogical and theoretical value. In this paper, we formally state, prove, and unify the Gauss-Landau Theorem for GCD and LCM, showing their deep symmetry and practical implications.

\section{Ancient Background}
The inquiry into the nature of divisibility is among the most venerable pursuits in number theory, reaching back to the ancient Greek tradition. In his monumental treatise, \emph{Elements}, Book VII, Euclid~\cite{Euclid} lays the foundation of what would later become the arithmetic of integers. There, he introduces the notion of a \emph{number} as a multitude composed of units:

\begin{proposition}[Euclid, Elements VII.1]
A number is a multitude composed of units.
\end{proposition}

Subsequently, Euclid characterizes \emph{prime numbers}—those indivisible by any other number save the unit itself:

\begin{proposition}[Euclid, Elements VII.2]
A prime number is that which is measured by a unit alone.
\end{proposition}

In these early propositions, Euclid articulates the core principles that govern the theory of divisibility. Most profoundly, he presents what undergirds the modern \emph{Euclidean Algorithm}:

\begin{proposition}[Euclid, Elements VII.3]
If a number divides two others, it also divides their difference.
\end{proposition}

This recursive method, elegant in its simplicity, permits the computation of the \emph{greatest common divisor} (gcd) of two integers. Given natural numbers \( a \) and \( b \) with \( a > b \), the algorithm proceeds by repeated application of the identity
\[
\gcd(a, b) = \gcd(b, a \bmod b),
\]
Until the remainder becomes zero, the last nonzero remainder is the desired greatest common measure. 

\vspace{2mm}
While the concept of the greatest common divisor is made explicit in Euclid's \emph{Elements}, the notion of the \emph{least common multiple} does not appear by name. Nevertheless, Euclid introduces ideas that are conceptually adjacent, especially in his treatment of ratios and common measures. For instance, in Book VII, Proposition 34, he shows how to find the least numbers that have the same ratio as two given numbers:

\begin{proposition}[Euclid, Elements VII.34]
To find the least numbers which have the same ratio as any given numbers.
\end{proposition}

This proposition seeks to reduce a given ratio to its simplest terms. For example:
\begin{problem}
Given the numbers \( a = 8 \) and \( b = 12 \). Determine the least numbers which are in the same ratio as \( a \) and \( b \). 
\end{problem}

\begin{proof} 
Given the numbers \( a = 8 \) and \( b = 12 \), the ratio \( 8 : 12 \) can be simplified to \( 2 : 3 \). These integers, \( 2 \) and \( 3 \), are the least numbers having the same ratio as \( 8 \) and \( 12 \), since any other pair with that same ratio (like \( 4 : 6 \) or \( 10 : 15 \)) would be multiples of them.
\end{proof}

This focus on reducing to minimal terms reflects Euclid’s concern with the fundamental structure of numerical relationships—an idea that, while not the least common multiple itself, conceptually foreshadows it through its emphasis on proportion, divisibility, and minimality.

\vspace{2mm}

In the 19th century, Carl Friedrich Gauss extended these ideas in his \emph{Disquisitiones Arithmeticae} \cite{Gauss}. He systematized number theory and provided the first rigorous proof of the Fundamental Theorem of Arithmetic: 
\begin{theorem}
    Every positive integer $n > 1$ has a unique prime factorization, up to order.
\end{theorem}

\begin{proof}[Proof (Gauss's First Proof)]
    
We divide the proof into two parts: existence and uniqueness.
\paragraph{Existence.} We prove by induction that every integer $n > 1$ can be factored into a product of primes. If $n$ is prime, the claim is trivially satisfied. Otherwise, $n$ is composite and can be written as $n = ab$, where $1 < a < n$ and $1 < b < n$. By the inductive hypothesis, both $a$ and $b$ can be factored into primes. Then, their product $n$ is also a product of primes. This completes the proof of existence.
\paragraph{Uniqueness.} Suppose $n > 1$ admits two distinct prime factorizations:
\[
n = p_1 p_2 \cdots p_k = q_1 q_2 \cdots q_l
\]
where all $p_i$ and $q_j$ are primes. We aim to show that the multisets $\{p_1, \ldots, p_k\}$ and $\{q_1, \ldots, q_l\}$ are equal. We proceed by induction on $n$. If $n$ is prime, the only possible factorization is $n$ itself, and uniqueness holds. Now assume uniqueness holds for all integers less than $n$. Consider the two factorizations of $n$ as above. Then $p_1$ divides the product $q_1 \cdots q_l$. Since $p_1$ is prime, it must divide one of the $q_j$, say $q_m$. But $q_m$ is also prime, so $p_1 = q_m$. Without loss of generality, we may reorder so that $p_1 = q_1$. Canceling this common prime from both sides, we obtain:
\[
\frac{n}{p_1} = p_2 \cdots p_k = q_2 \cdots q_l
\]
By the inductive hypothesis, the factorizations of $n/p_1$ are identical up to order. Therefore, the original factorizations of $n$ must also be the same up to order.
\end{proof}

After demonstrating the above, it was possible to formally define that every number $a_i$ has a unique prime factorization, which was further defined as follows

\begin{definition}
Let $a_1, a_2, \dots, a_n$ be a finite set of non-zero integers. Each $a_i$ has a unique prime factorization:
\[
    a_i = \prod_{j=1}^k p_j^{\beta_{ij}},
\]
where $p_j$ are the primes involved in at least one $a_i$, and $\beta_{ij} \in \mathbb{N}_0$ is the exponent of $p_j$ in $a_i$. If $p_j$ does not appear in $a_i$, we define $\beta_{ij} = 0$.
\end{definition}

\section{A Forgotten Theorem}
A lesser-known but mathematically striking identity emerges as an indirect consequence of the Fundamental Theorem of Arithmetic—one which, to the best of the author's knowledge, has not been explicitly recorded in the standard canon of number theory.
\vspace{2mm}

\subsection{Tracing the Theorem’s Foundations}

Let \( a_1, a_2, \dots, a_n \in \mathbb{N} \). By the \textbf{Fundamental Theorem of Arithmetic}, Each \( a_i \in \mathbb{N} \) has a unique factorization:
\[
a_i = \prod_{j=1}^{k} p_j^{\beta_{i,j}}
\]
Suppose a number \( d \in \mathbb{N} \) divides all \( a_i \). Then its prime factorization must be of the form:
\[
d = \prod_{j=1}^{k} p_j^{\delta_j} \quad \text{with } \delta_j \leq \beta_{i,j} \text{ for all } i
\]
In order to be the \textbf{greatest} such divisor, we must take the \textbf{largest possible} values of \( \delta_j \) that still satisfy the inequality, which gives:
\[
\delta_j = \min(\beta_{1,j}, \beta_{2,j}, \dots, \beta_{n,j})
\]
Hence:
\[
\gcd(a_1, a_2, \dots, a_n) = \prod_{j=1}^{k} p_j^{\min(\beta_{1,j}, \dots, \beta_{n,j})}
\]
Suppose a number \( m \in \mathbb{N} \) is divisible by all \( a_i \), i.e., \( a_i \mid m \) for all \( i \). Then its prime factorization must satisfy:

\[
m = \prod_{j=1}^{k} p_j^{\mu_j} \quad \text{with } \mu_j \geq \beta_{i,j} \text{ for all } i
\]
To be the \textbf{smallest} such multiple, we must choose:
\[
\mu_j = \max(\beta_{1,j}, \beta_{2,j}, \dots, \beta_{n,j})
\]
Therefore:
\[
\mathrm{lcm}(a_1, a_2, \dots, a_n) = \prod_{j=1}^{k} p_j^{\max(\beta_{1,j}, \dots, \beta_{n,j})}
\]

\subsection{The Two-Number Case}

In the particular case of two positive integers \( a \) and \( b \), a fundamental identity arises naturally from the prime factorizations of the numbers involved:
\[
a \cdot b = \gcd(a, b) \cdot \mathrm{lcm}(a, b)
\]
This identity is an immediate consequence of the additive structure of the exponents in their prime factorizations. Let:
\[
a = \prod_{j=1}^{k} p_j^{\beta_{1,j}}, \quad b = \prod_{j=1}^{k} p_j^{\beta_{2,j}},
\]
where each \( p_j \) is a prime dividing at least one of the two numbers, and the exponents \( \beta_{i,j} \in \mathbb{N}_0 \) follow the convention that \( \beta_{i,j} = 0 \) if \( p_j \nmid a_i \).
Then:
\[
\gcd(a, b) = \prod_{j=1}^{k} p_j^{\min(\beta_{1,j}, \beta_{2,j})}, \quad 
\mathrm{lcm}(a, b) = \prod_{j=1}^{k} p_j^{\max(\beta_{1,j}, \beta_{2,j})}
\]
Multiplying these two expressions together yields:
\[
\gcd(a, b) \cdot \mathrm{lcm}(a, b) = \prod_{j=1}^{k} p_j^{\min(\beta_{1,j}, \beta_{2,j}) + \max(\beta_{1,j}, \beta_{2,j})}
= \prod_{j=1}^{k} p_j^{\beta_{1,j} + \beta_{2,j}} = a \cdot b
\]

\subsection{Theorem Statement and Definition}

After Gauss, Edmund Landau \cite{Landau1903} was the first to systematically study the use of maximum and minimum exponents in prime factorizations, as noted by Dickson, who cites Landau's foundational work in 1919 \cite{Dickson1919}. Landau's contributions laid the essential groundwork for expressing the GCD and LMC in terms of prime exponents, establishing him as the primary figure credited with this aspect of number theory. For these reasons, the theorem presented here will be referred to as the \emph{Gauss-Landau Theorem} for identification in elementary number theory.
\vspace{2mm}
\begin{theorem}[Gauss-Landau Theorem]
    Let \(a_1, a_2, \dots, a_n \in \mathbb{N}\) be positive integers with unique prime factorizations
\[
a_i = \prod_{j=1}^k p_j^{\beta_{i,j}}, \quad \text{where } p_j \text{ are primes and } \beta_{i,j} \geq 0.
\]
Then the greatest common divisor of \(a_1, a_2, \dots, a_n\) is given by
\[
\gcd(a_1, a_2, \dots, a_n) = \prod_{j=1}^k p_j^{\min(\beta_{1,j}, \beta_{2,j}, \dots, \beta_{n,j})}.
\]
Moreover, the least common multiple of \(a_1, a_2, \dots, a_n\) is
\[
\mathrm{lcm}(a_1, a_2, \dots, a_n) = \prod_{j=1}^k p_j^{\max(\beta_{1,j}, \beta_{2,j}, \dots, \beta_{n,j})}.
\]
\end{theorem}

\section{Modern Background}
As previously mentioned, the earliest notions related to the Gauss-Landau theorem appear in the book \textit{History of the Theory of Numbers} by L. E. Dickson \cite{Dickson1919}, where he cites the works of E. Landau from 1903 and 1909. 
In his 1903 paper \cite{Landau1903}, Landau investigated the maximum order of permutations in the symmetric group \( S_n \). 
\begin{theorem}
Every permutation decomposes uniquely into disjoint cycles whose lengths sum to \( n \):
\[
a_1 + a_2 + \cdots + a_k = n, \quad a_i \in \mathbb{Z}^+.
\]
The order of such a permutation is the least common multiple of its cycle lengths:
\[
\operatorname{ord}(\pi) = \mathrm{lcm}(a_1, a_2, \dots, a_k).
\]
\end{theorem}

\noindent
In one of the earliest systematic treatments, Landau posed the problem of finding the maximum least common multiple among positive integers summing to \( n \),
\[
g(n) := \max_{\substack{a_1 + \cdots + a_k = n \\ a_i \geq 1}} \mathrm{lcm}(a_1, \dots, a_k),
\]
which corresponds to the maximal order of any permutation in \( S_n \) \cite[p.~47]{Dickson1919}.
\medskip
\begin{problem}
Consider \( n = 40 \) and a permutation \( \pi \in S_{40} \) whose cycle decomposition consists of two disjoint cycles of lengths 24 and 16, respectively. Determine the order of \( \pi \), i.e., the smallest positive integer \( m \) such that \( \pi^m = \mathrm{id} \).
\end{problem}
\begin{proof}[Solution]
   The order of a permutation \( \pi \in S_n \) is given by the least common multiple of the lengths of the disjoint cycles in its cycle decomposition. Here, the cycle lengths are 24 and 16, and since:
\[
24 = 2^3 \cdot 3, \quad 16 = 2^4,
\]
we have
\[
\mathrm{lcm}(24,16) = 2^{\max(3,4)} \cdot 3^{\max(1,0)} = 2^4 \cdot 3 = 48.
\]
Thus, \(\pi^{48} = \mathrm{id}\), and 48 is the order of \(\pi\).
\end{proof}

\medskip

In 1909, Landau \cite{Landau1909} further studied the maximum value of the LCM over all partitions of \( n \) into positive parts, formalizing and extending the problem from his 1903 work. He considered
\[
f(n) := \max_{\substack{a_1 + \cdots + a_k = n \\ a_i \geq 1}} \mathrm{lcm}(a_1, a_2, \dots, a_k),
\]
emphasizing the behavior of the maximal LCM and its prime factor exponents \cite[p. 52]{Dickson1919}.

\begin{definition}
Let \( n \in \mathbb{N} \). Define \( f(n) \) as the maximum value of the least common multiple of the parts in any partition of \( n \) into positive integers:
\[
f(n) := \max \left\{ \mathrm{lcm}(a_1, a_2, \dots, a_r) \;\middle|\; a_1 + a_2 + \cdots + a_r = n,\; a_i \in \mathbb{N} \right\}.
\] \end{definition}

\noindent
For instance, when \( n = 5 \), the set of all partitions into positive integers includes 
$$(5), (4+1), (3+2), (3+1+1), \dots.$$ 
Computing the least common multiple for each, we find:
\[
f(5) = \max \{5, \mathrm{lcm}(4,1), \mathrm{lcm}(3,2), \mathrm{lcm}(3,1,1), \dots \} = 6.
\]
Based on this, Landau established the following asymptotic behavior:
\[
\lim_{x \to \infty} \frac{\log f(x)}{\sqrt{x \log x}} = 1.
\]

\noindent
Let us solve this problem for \( n = 5 \), using both the classical version of the Gauss–Landau theorem from 1909 and its modern 2025 reformulation.

\begin{problem}
Let \( n = 5 \). List all the partitions of \( n \) into positive integers. For each partition, compute the least common multiple (LCM) of its parts.
\end{problem}

\begin{proof}[Classical Solution]
 We now compute Landau's example \( f(5) \), using modern notation and the Gauss--Landau classical theorem for the least common multiple. The function is defined as
\[
f(n) = \max \left\{ \mathrm{lcm}(a_1, a_2, \dots, a_r) \;\middle|\; a_1 + a_2 + \cdots + a_r = n,\; a_i \in \mathbb{N} \right\}.
\]

\noindent
In our case:
\[
f(5) = \max \left\{ \mathrm{lcm}(a_1, a_2, \dots, a_r) \;\middle|\; a_1 + \cdots + a_r = 5 \right\}.
\]

\noindent
We apply the \textbf{Gauss--Landau Classical Theorem}:
\[
\mathrm{lcm}(a_1, \dots, a_r) = \prod_{p \in \mathbb{P}} p^{\max(\beta_{1}, \dots, \beta_{r})}
\]
where \( \beta_i \) is the exponent of \( p \) in the prime factorization of \( a_i \).

\vspace{1em}

\begin{center}
\renewcommand{\arraystretch}{1.2}
\begin{tabular}{|c|c|c|c|}
\hline
\textbf{Partition} & \textbf{Prime Factorizations} & \textbf{Max Exponents} & \textbf{LCM} \\
\hline
\( (5) \) & \( 5^1 \) & \( 5^1 \) & 5 \\
\hline
\( (4,1) \) & \( 2^2, 1 \) & \( 2^2 \) & 4 \\
\hline
\( (3,2) \) & \( 3^1,\ 2^1 \) & \( 2^1 \cdot 3^1 \) & \textbf{6} \\
\hline
\( (3,1,1) \) & \( 3^1,\ 1,\ 1 \) & \( 3^1 \) & 3 \\
\hline
\( (2,2,1) \) & \( 2^1,\ 2^1,\ 1 \) & \( 2^1 \) & 2 \\
\hline
\( (2,1,1,1) \) & \( 2^1,\ 1,\ 1,\ 1 \) & \( 2^1 \) & 2 \\
\hline
\( (1,1,1,1,1) \) & \( 1,\dots,1 \) & none & 1 \\
\hline
\end{tabular}
\end{center}

\vspace{1em}

\noindent
We observe that the maximum value is:
\[
f(5) = \mathrm{lcm}(3,2) = 6.
\]
\end{proof}

\begin{proof}[Modern Solution]
Let \( a_1 = 3 \) and \( a_2 = 2 \), two positive integers with unique prime factorizations:
\[
\begin{aligned}
a_1 &= 3 = 2^0 \cdot 3^1 \cdot 5^0 \cdots, \\
a_2 &= 2 = 2^1 \cdot 3^0 \cdot 5^0 \cdots.
\end{aligned}
\]

\noindent
Let \( \{p_j\}_{j=1}^\infty = (2, 3, 5, 7, \dots) \) denote the sequence of all primes. For each \( a_i \), define its prime exponent vector
\[
\mathbf{v}_i = (\beta_{i,1}, \beta_{i,2}, \beta_{i,3}, \dots),
\]
where \( \beta_{i,j} \in \mathbb{Z}_{\geq 0} \) and \( a_i = \prod_{j=1}^\infty p_j^{\beta_{i,j}} \). Then:
\[
\begin{aligned}
\mathbf{v}_1 &= (0, 1, 0, 0, \dots), \\
\mathbf{v}_2 &= (1, 0, 0, 0, \dots).
\end{aligned}
\]
\noindent
By the \textbf{Gauss--Landau Modern Theorem}, we compute:
\[
\begin{aligned}
\mathrm{lcm}(a_1, a_2) &= \prod_{j=1}^\infty p_j^{\max(\beta_{1,j}, \beta_{2,j})}
= 2^{\max(0,1)} \cdot 3^{\max(1,0)} = 2^1 \cdot 3^1 = 6.
\end{aligned}
\]
\end{proof}
\noindent
In 1938, Hardy and Wright \cite{Hardy1938} incorporated Landau’s exact exponented‑prime formulations, making them standard in elementary number theory. They present:
\[
(n, m) = \prod_{j=1}^{k} p_j^{\min(i_j, j_j)}, \qquad [n, m] = \prod_{j=1}^{k} p_j^{\max(i_j, j_j)}.
\]
From this, one immediately derives the classical identity:
\[(n, m) \cdot [n, m] = n \cdot m, \] 
Clearly grounded in Landau’s approach.

\vspace{2mm}
\noindent
Later, in 1958, Landau E., \cite{Landau1958} published a book titled Elementary Number Theory, in which he proposed a problem that can be easily solved using the Gauss-Landau Theorem.

\begin{problem}[Landau, 1958]
Let \( a, b, c \in \mathbb{N} \). Prove the following identity:
\[
\gcd\big(\mathrm{lcm}(a,b),\ \mathrm{lcm}(b,c),\ \mathrm{lcm}(a,c)\big) = 
\mathrm{lcm}\big(\gcd(a,b),\ \gcd(b,c),\ \gcd(a,c)\big).
\]
\end{problem}
\begin{proof}[Solution]
We will prove the identity using the modern prime-exponent notation associated with the Gauss–Landau theorem. Let \( a, b, c \in \mathbb{N} \), and express them via their canonical prime factorizations:
\[
a = \prod_p p^{\alpha_p}, \quad b = \prod_p p^{\beta_p}, \quad c = \prod_p p^{\gamma_p},
\]
where the exponents \( \alpha_p, \beta_p, \gamma_p \in \mathbb{N}_0 \) and are zero for all but finitely many primes \( p \).
\medskip

\noindent
It is well known that for any two positive integers \( x = \prod_p p^{x_p} \) and \( y = \prod_p p^{y_p} \), we have:
\[
\gcd(x, y) = \prod_p p^{\min(x_p, y_p)}, \quad 
\mathrm{lcm}(x, y) = \prod_p p^{\max(x_p, y_p)}.
\]

\noindent
Applying this to the given identity, we analyze the exponent at each prime \( p \). The exponent on the left-hand side becomes:
\[
\min\Big\{ \max(\alpha_p, \beta_p),\ \max(\beta_p, \gamma_p),\ \max(\alpha_p, \gamma_p) \Big\},
\]
while the exponent on the right-hand side is:
\[
\max\Big\{ \min(\alpha_p, \beta_p),\ \min(\beta_p, \gamma_p),\ \min(\alpha_p, \gamma_p) \Big\}.
\]
\noindent
Therefore, the identity holds if and only if for all primes \( p \),
\[
\min\big\{ \max(\alpha_p, \beta_p),\ \max(\beta_p, \gamma_p),\ \max(\alpha_p, \gamma_p) \big\}
=
\max\big\{ \min(\alpha_p, \beta_p),\ \min(\beta_p, \gamma_p),\ \min(\alpha_p, \gamma_p) \big\}.
\]
Hence, the original identity follows by taking the product over all primes \( p \in \mathbb{P} \).
\end{proof}

In modern textbooks, such as those by Rosen \cite{Rosen}, Burton \cite{Burton}, Jones R., \cite{Jones2010}, and A., Smith \cite{Smith2015}
this technique appears in exercises but is rarely formalized. Students are often shown how to calculate $\gcd(60, 90)$ by factoring:
\begin{align*}
60 &= 2^2 \cdot 3^1 \cdot 5^1, \\
90 &= 2^1 \cdot 3^2 \cdot 5^1,
\end{align*}
and then taking the minimum exponents of the common primes:
\[
\gcd(60,90) = 2^1 \cdot 3^1 \cdot 5^1 = 30.
\]

\section{Discussion and Applications}

Although this result is implied in more advanced contexts, it is seldom articulated in elementary curricula. Assigning a name such as the \emph{Gauss–Landau Theorem} can help students recall and apply the concept more effectively in problems involving multiple integers.

\section*{Acknowledgements}
We thank the mathematical community for their ongoing support and insight. This work was inspired by a desire to clarify foundational number-theoretic principles for students.

\end{document}